\documentclass[12pt,sloppy]{article}
\usepackage{graphicx}
\usepackage{subfigure}
\usepackage{latexsym}
\usepackage{amsfonts}
\usepackage{amsthm}
\usepackage{amsmath}
\usepackage{chngpage}
\def\supertiny{\fontsize{4pt}{3pt}}

\author{{\bf Author: Maya Mohsin Ahmed} \\ Email: maya.ahmed@gmail.com }
\title{{\LARGE {\bf Two different scenarios when the  Collatz Conjecture fails }}}
\date{}

\newtheorem{thm}{Theorem}[section]

\newtheorem{lemma}{Lemma}[section]

\newtheorem{exm}{Example}[section]

\newtheorem{corollary}{Corollary}[section]

\begin{document}     
\maketitle           

\begin{abstract}

 In this article, we give two different proofs of why the Collatz Conjecture
 is false.

\end{abstract}

\section{Introduction.}

Given a positive integer $A$, construct the sequence $c_i$ as follows:
\[  
\begin{array}{llll}
c_i & = & A & \mbox{if $i=0$;} \\ & = & 3c_{i-1}+1 & \mbox{if
  $c_{i-1}$ is odd;} \\ & = & c_{i-1}/2 & \mbox{if $c_{i-1}$ is even.}
\end{array} 
\]

The sequence $c_i$ is called a {\em Collatz sequence} with {\em
  starting number} $A$. The {\em Collatz Conjecture} says that this
sequence will eventually reach the number 1, regardless of which
positive integer is chosen initially. The sequence gets in to an
infinite cycle of 4, 2, 1 after reaching 1.

\begin{exm} \label{Collatzeg} {\em 
The Collatz sequence of $911$ is:
\[
\begin{array}{l}

911, 2734, 1367, 4102, 2051, 6154,  3077, 9232,
4616, 2308, 1154, 577,\\ 1732, 866, 433, 1300, 650, 325, 976, 488, 244,
122, 61, 184, 92,  46, 23, \\ 70, 35, 106, 53, 160, 80, 40, 20, 10, 5, 16,
8, 4, 2, 1,4,2,1,4,2,1, \dots
\end{array}
\]
} \end{exm}

For the rest of this article, we will ignore the infinite cycle of
$4,2,1$, and say that a Collatz sequence {\em converges to $1$}, if it
reaches $1$.  A comprehensive study of the Collatz Conjecture can be
found in \cite{lag1}, \cite{lag2}, and \cite{wir}.

In this article, as we did before in \cite{ahmed1}, \cite{ahmed2}, and
\cite{ahmed3}, we focus on the subsequence of odd numbers of a Collatz
sequence. This is because every even number in a Collatz sequence has
to reach an odd number after a finite number of steps.  Observe that
the Collatz Conjecture implies that the subsequence of odd numbers of
a Collatz sequence converges to $1$.

In Section \ref{proofsection}, we use networks of Collatz sequences
\cite{ahmed2} to prove that the Collatz Conjecture fails. In Section
\ref{doNotConvergeSection}, we use the notion of Reverse Collatz
sequences \cite{ahmed1} to give another proof of the collapse of the
Collatz Conjecture.


\section{Networking to prove that the Collatz Conjecture is false.} \label{proofsection}

In this section, we use an array of Collatz sequences to demonstrate
how the Collatz Conjecture fails.  In \cite{ahmed2}, we proved the
following theorem which showed that the Collatz sequence of an odd
number $A$ merges either with the Collatz sequence of $(A-1)/2$ or
$2A+1$.

\begin{thm} \label{timefor2plythm} (Theorem 2.1, \cite{ahmed2}) 
Let $N$  be an odd  number. Let $n_0  = N$, $m_0 =  2n_0+1 =
2N+1$, and $l_0 = 2m_0+1$ = 4N+3 . Let $n_i, m_i$, and $l_i$
denote  the  subsequence  of  odd  numbers  in  the  Collatz
sequence of  $n_0, m_0,$ and $l_0$,  respectively. Then, for
some integer $r$, $n_{r+1}  = (3n_{r}+1)/2^k$ such that $k >
1$. Let $r$  be the smallest such integer.  Then, $m_{r+2} =
(3m_{r+1}+1)/2^j$ for some $j > 1$, and

\[ \begin{array}{l} 
m_i =  2n_i+1, \mbox{ for }  i \leq r, \\
 m_{{r}+1} = 2^kn_{{r}+1}+1 \\
l_i = 2m_i+1 \mbox{ for } i \leq r +1, \\
l_{r +2} = 2^jm_{r +2}+1
\end{array}
\]

If $k=2$, then $m_i = n_i$ for $i > r +1$. Otherwise, if
$k>2$ then $l_{r +2} = 4m_{r +2}+1$ and $l_i = m_i$ for $i > r+2$.
\end{thm}

For some integer $u_0$, consider the sequence $u_i = 2u_{i-1}+1$, for
$i \geq 1$. Let $v_{j,i}$ denote the sequence of odd integers in the
Collatz sequence of $u_i$.  Theorem \ref{timefor2plythm} tells us that
for every $i > 0$, there is some $r$ such that for $j \leq r$,
$v_{j,i} = 2v_{j,i-1}+1$ and $v_{r+1,i} = 2^kv_{r,i-1}+1$ where
$k>1$. This fact motivates us to construct the array of Collatz
sequences in Theorem \ref{networkthm}.

Let $A$ be an odd
number. Write $A$ in its binary form,
\begin{align*}  \begin{array}{lllllllllll}
A &=& 2^{i_1} + 2^{i_2} + \cdots + 2^{i_m}+ 2^n+2^{n-1} + 2^{n-2} + \cdots + 2^2+ 2+1,\\ 
& & \hfill  \mbox{such that } i_1 > i_2 > \cdots > i_m > n+1.
\end{array}
\end{align*} 
 The {\em tail} of $A$ 
is defined as $2^n+2^{n-1} + 2^{n-2} + \cdots + 2^2+ 2+1$. We call $n$ the {\em
  length of the tail} of $A$ (See \cite{ahmed1}). 

\begin{exm}{\em
The tail of $27 = 2^4 + 2^3 + 2 + 1$ is $2+1$ and hence has length $1$,
the tail of $161 = 2^7 + 2^5 + 1$ is $1$ and hence is of length $0$,
and the tail of $31 = 2^4 + 2^3 + 2^2 + 2 + 1$ is the entire number
$2^4 + 2^3 + 2^2 + 2 + 1$ and therefore has length $4$.  }
\end{exm}

 In \cite{ahmed1} we proved the following theorem about the odd
 numbers in a Collatz sequence.
\begin{thm}  \label{tailsthm} (Theorem 2 \cite{ahmed1})
Let $A$ be an odd number and let $n$ denote the length of the tail of
$A$.  Let $a_i$ denote the sequence of odd numbers in the Collatz
sequence of $A$ with $a_0 = A$.
\begin{enumerate}

\item \label{tailis0thmpart} If $n \geq 1$, then  for some 
$i_1>i_2 > \cdots > i_m > n+1$,
\[ a_i = \frac{3a_{i-1}+1}{2} = \frac{3^i}{2^i} (2^{i_1} +
  2^{i_2} + \cdots + 2^{i_m}+2^{n+1}) -1, \mbox{ for } i=1, \dots
  n.
\]
 The length of the tail of $a_i$ is $n-i$. Hence the length of the
 tail of the $n$-th odd number after $A$ is $0$.

\item If $n=0$, then 
\[ a_1  = \frac{3A+1}{2^k}, k \geq 2. \]
\end{enumerate}
 
\end{thm}

\begin{corollary} \label{nextTermCollatzcorol}
Let $A$ be an odd number. If $A \not \equiv 1$  mod $4$, then the next odd term in the Collatz sequence of $A$ is $(3A+1)/2$.
\end{corollary} 

\begin{proof}

Since $A$ is odd and $A \not \equiv 1$ mod $4$, $A \equiv 3$ mod
$4$. This implies that the tail of $A$ has length greater than
zero. Hence the proof follows from Part \ref{tailis0thmpart} of
Theorem \ref{tailsthm}.

\end{proof}

\begin{thm} \label{networkthm} (Theorem 5.1, \cite{ahmed2})
For $n  \not \equiv  1$ mod $3$,  define a diagonal  array as
follows. Let  $u_0 =  4n+1$ and $u_i  = 2u_{i-1}+1$,  for $i
\geq 1$. For $j \geq  0$, let $v_{0,j}=u_j$, and for $k \geq
1$, let $v_{k,k} = 3v_{k-1,k-1}+2$. Finally, for $j >i$, let
$v_{i,j} = 2v_{i,j-1}+1$. We get an array
 \[\begin{array}{lllllllllllllllll}
u_0 & u_1 & u_2 &u_3 & u_4 & u_5 &  u_6 & u_7 & \dots \\
& v_{1,1} & v_{1,2} & v_{1,3} & v_{1,4} & v_{1,5} & v_{1,6} & v_{1,7} & \dots \\
&& v_{2,2} & v_{2,3} &  v_{2,4} & v_{2,5} & v_{2,6} & v_{2,7} & \dots \\
&&& v_{3,3} & v_{3,4} &  v_{3,5} & v_{3,6} & v_{3,7} &  \dots \\
&&&& \dots & \dots  \\

\end{array}
\]
with the following properties:

\begin{enumerate}

\item \label{partu_0multi3}  $u_k \not  \equiv 2$  mod $3$  for all  $k \geq 0$, whereas,
  $v_{i,j} \equiv 2$ mod $3$, if $i>0$ and $j>0$.

\item \label{partnotjump} $u_0 \equiv 1$ mod  $4$ and $v_{i,i} \equiv 1$ mod $4$
  for all  $i$. $u_i  \not \equiv 1$  mod $4$ for  $i>0$ and
  $v_{i,j} \not \equiv 1$ mod $4$ if $i \not = j$.

\item For $j \geq 1$, the $j$-th column is the first few odd
  numbers in the Collatz sequence of $u_j$.

\item  For  $j \geq  i >0$,
   $v_{i,j} = 3v_{i-1,j-1}+2$.

\end{enumerate}

\end{thm}

\begin{proof}
\begin{enumerate}

\item Since $n \not \equiv 1$ mod $3$, $u_0=4n+1 \not \equiv
  2$ mod $3$. If $u_i  \not \equiv 2$ mod $3$, then $u_{i+1}
  = 2u_i+1  \not \equiv 2$ mod $3$.  Consequently, $u_k \not
  \equiv  2$ mod  $3$  for all  $k  \geq 0$.  

For $i >  0$, $v_{i,i} \equiv 2$ mod  $3$, by definition. If
$v_{i,j}  \equiv 2$  mod $3$,  then $v_{i,j+1}  = 2v_{i,j}+1
\equiv 2$ mod $3$. Thus,  it follows that $v_{i,j} \equiv 2$
mod $3$, if $i>0$ and $j>0$.

\item    $u_0   \equiv    1$   mod    $4$    by   definition.
  $v_{i,i}=3v_{i-1,i-1}+2 \equiv 1$  mod $4$, if $v_{i-1,i-1}
  \equiv 1$ mod $4$. Since  $v_{0,0} = u_0$, it follows that
  $v_{i,i} \equiv 1$ mod $4$ for $i \geq 0$. For $i>0$, $u_i
  =2u_{i-1}+1 \not  \equiv 1$ mod $4$ since  $u_{i-1}$ is an
  odd number. Similarly,  $v_{i,j} =2v_{i,j-1}+1 \not \equiv
  1$ mod $4$, when $i \not = j$.

\item When $j \geq 1$, $v_{0,j}=u_j$. We know from Part
  \ref{partnotjump} that $v_{i,j} \equiv 3$ mod $4$ if $i \not = j$.
  Therefore, by Corollary \ref{nextTermCollatzcorol}, the odd number
  that comes after $v_{i,j}$ in the Collatz sequence of $v_{i,j}$ is
  $(3v_{i,j}+1)/2$. We do not compute the odd numbers that come after
  $v_{i,i}$.  Hence, assume that $j>i$.

For $k  \geq 1$, let $z_k$  be defined as follows:  \[ z_k =
2^k + 2^{k-1}  + 2^{k-2} + \cdots +  2+1 = \sum_{i=0}^k2^i =
\frac{2^{k+1}-1}{2-1} = 2^{k+1}-1.\]
Then, by definition, for $j >i$,  $v_{i,j} = 2^{j-i}v_{ii}+z_{j-i-1}$, and
\[
v_{i+1,j} = 2^{j-i-1}v_{i+1,i+1}+z_{j-i-2} = 2^{j-i-1}(3v_{i,i}+2) +z_{j-i-2}.
\]
\[
\frac{3v_{i,j}+1}{2} = \frac{3\times2^{j-i}v_{ii}+3\times z_{j-i-1}+1}{2} = 
3\times2^{j-i-1}v_{ii}+ \frac{3z_{j-i-1}+1}{2}.\]

\[\frac{3z_{j-i-1}+1}{2} = \frac{3(2^{j-i}-1)+1}{2} =  3\times2^{j-i-1}-1 =
2 \times 2^{j-i-1} + (2^{j-i-1}-1) = 2^{j-i}+z_{j-i-2}.\]

Consequently,
\[
\frac{3v_{i,j}+1}{2} = v_{i+1,j}.\] This implies $v_{i+1,j}$
is the odd number that  comes after $v_{i,j}$ in the Collatz
sequence  of $v_{i,j}$. Hence,  for $j  \geq 1$,  the $j$-th
column is the first few  odd numbers in the Collatz sequence
of $u_j$.

\item For $j \geq i >0$, when $i=j$, we have that $v_{i,j} =
  3v_{i-1,j-1}+2$, by definition. Let $i \not = j$, then $v_{i,j} =
  2v_{i,j-1}+1$. But $v_{i-1,j-1} \not \equiv 1$ mod $4$, when $i \not
  = j$.  Hence, by Corollary \ref{nextTermCollatzcorol}, $v_{i,j-1} =
  (3v_{i-1,j-1}+1)/2$.  Consequently, $v_{i,j} = 3v_{i-1,j-1}+2$.

\end{enumerate}
\end{proof}

Theorem \ref{networkthm} can be used to construct divergent Collatz
sequences as shown in the example below. Observe that the Collatz
sequence of $u_i$ is strictly increasing till $v_{i,i}$. Our choice of
$u_0$ and Theorem \ref{timefor2plythm} makes sure this is the
case. Check that $v_{i+1,i}$ will be smaller than $v_{i,i}$. We see
that as $i$ increases, a subsequence of the Collatz sequence of $u_i$
is increasing indefinitely.  Thus creating divergent Collatz
sequences. This implies that the Collatz Conjecture is
false. Different values of $n$ provide different arrays that lead to
different divergent Collatz sequences.

\begin{exm}{\em Let $u_i$ and $v_{i,j}$ be defined as in Theorem  \ref{networkthm}. 

$n=0:$

\[\begin{array}{lllllllllllllllllllllllllllllllllllllllll} 
u_i:  & 1 & 3 & 7 & 15 & 31 & 63 & 127 & 255 & 511 & 1023 & 2047 & 4095 & 8191  \\ 
v_{1,i}:& &5 & 11  & 23 & 47 &  95 & 191 & 383 &767  & 1535 &3071 & 6143 & 12287 \\ 
v_{2,i}: & & & 17 & 35 &71 & 143 & 287 &  575 & 1151 & 2303 &4607 & 9215  & 18431\\ 
v_{3,i}: && & &  53 & 107 & 215 & 431 & 863 & 1727 & 3455 & 6911 & 13823 & 27647\\
v_{4,i}: && & & & 161 & 323 & 647 & 1295 & 2591 & 5183 & 10367 &20735 &41471 \\
v_{5,i}&&&&&& 485 & 971 & 1943 & 3887 & 7775 &15551 &31103 &62207 \\
v_{6,i}: &&&&&&& 1457 & 2915 & 5831 &11663 & 23327 &46655 &93311 \\
v_{7,i}: &&&&&&&& 4373 & 8747& 17495 & 34991&  69983 & 139967 \\
\end{array}
\]

$n=3:$
\[\begin{array}{lllllllllllllllllllllllllllllllllllllllll} 

u_i: & 13 &    27 & 55 & 111 & 223 & 447&895& 1791&3583& 7167&14335&28671 \\
v_{1,i}: &  & 41 & 83 & 167 & 335 & 671 & 1343 & 2687 & 5375 & 10751 & 21503 & 43007\\ 
v_{2,i}: & & &  125 & 251 & 503 & 1007 & 2015 & 4031 & 8063 & 16127 & 32255 & 64511 \\
v_{3,i}:&  & &  & 377 & 755 & 1511 &  3023 & 6047 &  12095 & 24191 &  48383 & 96767 \\
v_{4,i}: & &&&&1133 & 2267 & 4535 & 9071 & 18143 & 36287 &  72575 &  145151 \\
v_{5,i}: &  &&&&&   3401 &    6803 & 13607 &  27215 &  54431 &  108863 & 217727 \\ 
v_{6,i}: & &&&&&& 10205& 20411& 40823&81647&163295&326591

\end{array}
\]
}  \end{exm}

\section{Reversing to prove that the Collatz Conjecture is false} \label{doNotConvergeSection}
In this section, we provide a different proof of how the Collatz
Conjecture fails.  

Let $A$ be an odd integer. We say $A$ is a {\em jump}, if $A=4n+1$
from some odd number $n$. If $A = 4^i \times P + 4^{i-1}+4^{i-2} +
\cdots+4 + 1$, such that $i \geq 1$ and $P$ is an odd number, then we
say $A$ is a {\em jump from $P$ of height $i$}.
\begin{exm} {\em
$13 = 4 \times 3 + 1$ is a jump from $3$ of height $1$. $53 = 4 \times
    13 +1 = 4^2 \times 3 +4 + 1$ is a jump from $13$ of height $1$ and
    a jump from $3$ of height $2$.  }
\end{exm}

Jumps are studied in great detail in \cite{ahmed1} and
\cite{ahmed2}. We say two Collatz sequences are {\em equivalent} if
the second odd number occurring in the sequences are same.

\begin{exm} {\em
The Collatz sequence of 3 is
\[
3, 10, 5, 16, 8, 4, 2, 1, 1, \dots \]
The Collatz sequence of 13 is
\[
13, 40, 20, 10, 5, 16, 8, 4, 2, 1, 1, \dots 
\]
Observe that the two sequences merge at the odd number $5$.  Hence the
Collatz sequences of $3$ and $13$ are equivalent.  }\end{exm}

\begin{lemma}[Corollary 2, Section 2, \cite{ahmed1}] \label{lemmajump}
Let $A$ be an odd number and let $c_0=A$ and $c_i = 4c_{i-1}+1$, that
is, $c_i$ are jumps from $A$. Then, for any $i$, the Collatz sequence
of $A$ and $c_i$ are equivalent.
\end{lemma}

The Reverse Collatz sequence, $r_i$, of a positive integer $A$ was
 defined in \cite{ahmed1} as follows.

 \[
r_i = \left \{  \begin{array}{llllllll} A  & \mbox{ if $i = 0$; } \\

 \frac{r_{i-1}-1}{3} & \mbox{ if $r_{i-1} \equiv 1$ mod $3$ and $r_{i-1}$ is
  even}; \\ 
2r_{i-1} &  \mbox{ if $r_{i-1} \not \equiv 1$ mod $3$ and
    $r_{i-1}$ is even;} \\ 2r_{i-1} &  \mbox{ if $r_{i-1}$ is odd}.
\end{array}
\right .
\]

We say that a Reverse Collatz sequence {\em converges} if the
 subsequence of odd numbers of the sequence converges to a multiple of
 $3$.

\begin{exm} \label{reveresecollarzexamp} \hfill {\em

The Reverse Collatz sequence with starting number  $121$ is :
\begin{gather*}
121, 242, 484, 161, 322, 107, 214, 71, 142, 47, 94, 31, 62, 124, 41,
82, 27, 54, 108, 216, \dots
\end{gather*}
The Reverse Collatz sequence of $121$ converges because its
subsequence of odd numbers
\[
121, 161, 107, 71, 47, 31,  41, 27
\]
converges to $27$. 

}
\end{exm}

Let $p_i$ denote the subsequence of odd numbers in the Reverse Collatz
sequence of $A$. Then, in \cite{ahmed1}, it was proved that, if $p_i
\equiv 0$ mod $3$, then $p_{i+1}$ do not exist. Otherwise, $p_{i+1}$
is the smallest odd number before $p_i$ in any Collatz sequence and

\begin{equation} \label{Reversecollapieqn}
p_{i+1} = \left \{ \begin{array}{l}
\frac{2p_i-1}{3} \mbox{ if } p_i \equiv 2 \mbox{ mod } 3 \\ \\
\frac{4p_i-1}{3} \mbox{ if } p_i \equiv 1 \mbox{ mod } 3 
\end{array}
\right .
\end{equation}
 
It was conjectured in \cite{ahmed1}, that, the Reverse Collatz
sequence converges to a multiple of $3$ for every number greater than
one. See \cite{ahmed1} and \cite{ahmed2} for more details about
Reverse Collatz sequences.

\begin{lemma} \label{unwindAndReverseLemma}
Let $A$ be an odd number such that $A \equiv 2 \mod{3}$.  Let $u =
(A-2)/3$.  If $r_i$ denotes subsequence of odd numbers in the Reverse
Collatz sequence of $A$ with $r_0=A$, then $r_1 = 2u+1$. Consequently, 
\[u \equiv \left \{ \begin{array}{lllllll} 
0 \mod{3},  \implies  r_1 \equiv 1 \mod{3} \\
1 \mod{3},  \implies  r_1 \equiv 0 \mod{3} \\
2 \mod{3},  \implies  r_1 \equiv 2 \mod{3} \\
\end{array}
\right .
\]
Observe that $r_1 < r_0$. Moreover, if $u \equiv 2 \mod{3}$, let $t_i$
represent the subsequence of odd numbers in the Reverse Collatz
sequence of $u$, then $t_1= (r_1-2)/3$.
\end{lemma}
\begin{proof}
 Since $A \equiv 2 \mod{3}$, by definition of Reverse
Collatz sequence, $r_1 = (2A-1)/3$.  Now 
\[2u+1 = 2 \left ( \frac{A-2}{3} \right ) + 1 = \frac{2A-1}{3}. \] 

Therefore, $r_1 = 2u+1$.  Consequently, \[ \frac{r_1 - 2}{3} =
\frac{(2u+1)-2}{3} = \frac{2u-1}{3}. \] If $u \equiv 2 \mod{3}$, then
by definition of Reverse Collatz sequence, $t_1 = (2u-1)/3$. Thus,
$t_1 = (r_1-2)/3$. 
\end{proof}

\begin{thm} \label{unwindthm}

Let $A$ be an odd number and let $A \equiv 2 \mod{3}$.  Define a
sequence of odd numbers, $v_{0,j}$, such that, $v_{0,0}=A$, and for
$j>0$, $v_{0,j} = (v_{0,j-1}-2)/3$.  Then, for some integer $n \geq
0$, $v_{0,j} \equiv 2 \mod{3}$ for $j<n$, and $v_{0,n} \not \equiv 2
\mod{3}$. Moreover, there are at least $n+1$ terms in the
subsequence of odd numbers in the Reverse Collatz sequence of
$v_{0,0}$ (by Part \ref{reverseItem}).  Let $v_{i,0}$, $i=0, \dots, n$, denote the first $n+1$
terms of the subsequence of odd numbers in the Reverse Collatz
sequence of $v_{0,0}$. Then, for each $i = 1, \dots, n$, we can form
an array $v_{i,j} = (v_{i,j-1}-2)/3$, where $j=1, \dots, n-i$,

 \[\begin{array}{lllllllllllllllll}
v_{0,0} & v_{0,1} & v_{0,2} & v_{0,3} &  \dots &  v_{0,n-2} & v_{0,n-1}& v_{0,n}  \\
v_{1,0} & v_{1,1} & v_{1,2} & v_{1,3} &  \dots &  v_{1,n-2} & v_{1,n-1} \\
v_{2,0}& v_{2,1}& v_{2,2} & v_{2,3} &   \dots   & v_{2,n-2}       \\

\vdots& \vdots& \vdots& \vdots   \\ 
v_{n-3,0}&v_{n-3,1}& v_{n-3,2}& v_{n-3,3}   \\
v_{n-2,0}&v_{n-2,1}& v_{n-2,2}\\
v_{n-1,0}&v_{n-1,1}   \\
v_{n,0}

\end{array}
\]

with the following properties.

\begin{enumerate}

\item \label{reverseItem} For each $j=0, \dots n$, $v_{i,j}$, $i=0,1,
  \dots, n-j$, are the first $n+1-j$ terms of the subsequence of odd
  numbers in the Reverse Collatz sequence of $v_{0,j}$. For $i=0,
  \dots, n$, $v_{i,j} \equiv 2 \mod{3}$ whenever $j \not = n-i$, and
  $v_{i,n-i} \not \equiv 2 \mod{3}$.

\item \label{partrelate} For $i>0$, $v_{i,j} = 2*v_{i-1,j+1}+1$.  Moreover, if $v_{0,n}
  \equiv 1 \mod{3}$ then,
\[
v_{i,n-i} \equiv  \left \{
\begin{array}{llllllllll}

 0 \mod{3} &  \mbox{  if $i$ is odd; } \\
 1 \mod{3} &  \mbox{  if $i$ is even.}

\end{array}
\right .
\]
On the other hand, if $v_{0,n} \equiv 0 \mod{3}$ then,
\[
v_{i,n-i} \equiv  \left \{
\begin{array}{llllllllll}

 1 \mod{3} &  \mbox{  if $i$ is odd; } \\
 0 \mod{3} &  \mbox{  if $i$ is even.}

\end{array}
\right .
\]

\item  \label{startEndPart}  $v_{0,0} = 3^n \left( v_{0,n}+1 \right )-1$  and  $v_{n,0} = 2^n(v_{0,n}+1)-1$.  Moreover, for $i=1, \dots, n$,
$v_{i,0} = 3^{n-i} \times 2^i \times  \left( v_{0,n}+1 \right )-1$.

\end{enumerate}

\end{thm}

{\em Proof.}  

\begin{enumerate}
\item By Lemma \ref{unwindAndReverseLemma}, since $v_{0,1} \equiv 2
  \mod{3}$, $v_{1,0} \equiv 2 \mod{3}$, and $v_{1,1}$ is the next odd
  term in the subsequence of odd numbers in the Reverse Collatz
  sequence of $v_{0,1}$.  Now, because $v_{1,0} \equiv 2 \mod{3}$,
  $v_{2,0}$ is well defined as the next term in the Reverse Collatz
  sequence of $v_{1,0}$, by Equation \ref{Reversecollapieqn}.  Since,
  $v_{0,2} \equiv 2 \mod{3}$, we apply Lemma
  \ref{unwindAndReverseLemma}, again, to conclude that $v_{1,1} \equiv
  2 \mod{3}$, and $v_{1,2}$ is the next odd term in the subsequence of
  odd numbers in the Reverse Collatz sequence of $v_{0,2}$.

 Applying this argument, repeatedly, we derive that for each $j=1,
 \dots, n-1$, $v_{1,j}$ is the first odd term that comes after
 $v_{0,j}$ in the Reverse Collatz sequence of $v_{0,j}$.  We also get
 that $v_{1,j} \equiv 2 \mod{3}$, for $j=1, \dots, n-2$.  Since
 $v_{0,n} \not \equiv 2 \mod{3}$, $v_{1,n-1} \not \equiv 2 \mod{3}$,
 by Lemma \ref{unwindAndReverseLemma}.

Now, since we established that $v_{1,1} \equiv 2 \mod{3}$, we will
repeat the above argument to derive that for each $j=1, \dots, n-2$, 
$v_{2,j}$ is the first odd term that comes after $v_{1,j}$ in the
Reverse Collatz sequence of $v_{1,j}$.  We also get that $v_{2,j}
\equiv 2 \mod{3}$, for $j=1, \dots, n-3$, and  $v_{2,n-2} \not \equiv 2 \mod{3}$.

Continuing thus, we get for each $j=0, \dots n$, $v_{i,j}$, $i=0,1,
\dots, n-j$, are the first $n+1-j$ terms of the subsequence of odd
numbers in the Reverse Collatz sequence of $v_{0,j}$.  For $i=0,
\dots, n$, $v_{i,j} \equiv 2 \mod{3}$ whenever $j \not = n-i$, and
$v_{i,n-i} \not \equiv 2 \mod{3}$. 

\item This result follows from Lemma \ref{unwindAndReverseLemma}.

\item  Rewriting $v_{i,j} = (v_{i,j-1}-2)/3$, we get
  $v_{i,j-1} = 3v_{i,j}+2$.  Thus,
\[v_{i,0} = 3v_{i,1}+2 = 3(3(v_{i,2}+2)+2 = 3^2v_{i,2}+ 3\times2 + 2. \]
Continuing thus we get $v_{i,0} = 3^{n-i} \times v_{i,n-i} + 2 \times \sum_{s=0}^{n-i-1} 3^s$. Since 
\[
2 \times \sum_{s=0}^{n-i-1} 3^s = 3^{n-i}-1,
\] 
we get \[ v_{i,0} = 3^{n-i}  (v_{i,n-i}+1) -1. \] 

 By Part \ref{partrelate}, $v_{i,j} = 2*v_{i-1,j+1}+1$, for $i>0$.
 Therefore,
\[
v_{i,n-i} =  2*v_{i-1,n-i+1}+1 = 2*(2*v_{i-2,n-i+2}+1)+1
\]
Continuing thus we get
\[ v_{i,n-i} = 2^i \times v_{0,n} + \sum_{s=0}^{i-1} 2^s. \]
Substituting $\sum_{s=0}^{i-1} 2^s = 2^i - 1$, we get $v_{i,n-i} = 2^i
(v_{0,n} + 1) -1$. In particular, $v_{0,0} = 3^n \left( v_{0,n}+1
\right )-1$ and $v_{n,0} = 2^n(v_{0,n}+1)-1$.  

Since $v_{0,0} = 3^n \left( v_{0,n}+1 \right )-1$ and $v_{0,0} \equiv
2 \mod{3}$, $v_{1,0} = 3^{n-1} \times 2 \times \left( v_{0,n}+1 \right
)-1$, by Equation \ref{Reversecollapieqn}. By Part \ref{reverseItem},
$v_{1,0} \equiv 2 \mod{3}$. Therefore, again, by Equation
\ref{Reversecollapieqn}, we derive $v_{2,0} = 3^{n-2} \times 2^2 \times \left(
v_{0,n}+1 \right )-1$. Repeating this argument, we get, for $i=1,
\dots, n$, $v_{i,0} = 3^{n-i} \times 2^i \times \left( v_{0,n}+1
\right )-1$.

\qed
\end{enumerate}

\begin{exm}
{\em In this example, we apply Theorem \ref{unwindthm} to $A=2429
  \equiv 2 \mod{3}$. Here, $n=5$, $v_{0,5} = 9$,   The columns
  are the first odd terms in the Reverse Collatz sequence of
  $v_{0,j}$.  Observe that $v_{0,5}=9, v_{2,3}=39, v_{4,1}= 159$ are
  $\equiv 0 \mod{3}$ and $ v_{1,4}=19$, $v_{3,2}=79$, and $v_{5,0}=319$
  are $\equiv 1 \mod{3}$. $v_{0,0} = 2429 = 3^5
  \times 10 -1$, and $v_{5,0} = 2^5 \times 10 -1 = 319$.

\begin{adjustwidth}{-1in}{.5in}

\[ \begin{array}{llllllllllllllllll}
 v_{0,j}: & 2429  & 809 & 269 & 89 & 29 & 9  \\
v_{1,j}: &  1619  & 539  & 179 & 59 & 19 \\
v_{2,j}: &1079  & 359  & 119 & 39    \\
v_{3,j}: &719  & 239 & 79   \\
v_{4,j}: &479 &  159  \\
v_{5,j}: &319  

\end{array}
\]
\end{adjustwidth}

}
\end{exm}

\begin{lemma} \label{3n+1lemma}
Let $A \equiv 1 \mod{3}$ be an odd number. Then we can write $A = 3^n
B + 1$ such that $B$ is not divisible by $3$. If $r_i$ denotes the subsequence of odd numbers in the 
Reverse Collatz sequence of $A$, then, for $i=0, \dots n$, $r_i = 4^i \times 3^{n-i} \times B + 1$.
Thus, $r_n = 4^n \times  B + 1$. 
For $i=0, \dots n-1$, $r_i  \equiv 1  \mod{3}$, and
\[
r_n  \equiv \left \{ \begin{array}{lllllllllll}

2  \mod{3},  & \mbox{ if $B \equiv 1 \mod{3}$,} \\
0  \mod{3},  & \mbox{ if $B \equiv 2 \mod{3}$.} \\

\end{array}
\right .
\]
Observe that $r_i > r_{i-1}$ for $i=1, \dots, n$.

\end{lemma}

{\em Proof.} Since, $A \equiv 1 \mod{3}$, by Equation
\ref{Reversecollapieqn}, we get $r_1 = (4A-1)/3 = 4 \times 3^{n-1}B +
1$.  If $n>1$ then $r_1 \equiv 1 \mod{3}$. Again, by Equation
\ref{Reversecollapieqn}, $r_2 = 4^2 \times 3^{n-2}B + 1$. If $n>2$
then $r_2 \equiv 1 \mod{3}$. Continuing this argument, we get, for
$i=0, \dots n-1$, $r_i = 4^i \times 3^{n-i} \times B + 1$, such that,
$r_i \equiv 1 \mod{3}$. Now since $r_{n-1} \equiv 1 \mod{3}$, we get
$r_n = 4^n \times B + 1$.  Since $B \not \equiv 0 \mod{3}$,  $r_n \not \equiv 1 \mod{3}$. In fact, 
$r_n  \equiv 2  \mod{3}$,   if $B \equiv 1 \mod{3}$, and $r_n  \equiv 0  \mod{3}$,   if $B \equiv 2 \mod{3}$. \qed

\begin{exm}
{\em We apply Lemma \ref{3n+1lemma} to $A = 91$. We can write $A = 3^2
  \times 10 + 1$.  Let $r_i$ denote the subsequence of odd numbers in
  the Reverse Collatz sequence of $A$. Then,
\begin{adjustwidth}{-1in}{.5in}

\[ \begin{array}{llllllllllllllllll}

r_0= 91 = 3^2 \times 10 + 1 \\
r_1 = 121 = 4 \times 3 \times 10 + 1 \\
r_2 = 161 = 4^2 \times 10 + 1 \equiv 2 \mod{3}.
\end{array}
\]
\end{adjustwidth}
Since $10 \equiv 1 \mod{3}$, $r_2  \equiv 2  \mod{3}$.
}
\end{exm}

\begin{exm}
{\em In this example, we demonstrate the convergence of the Reverse
  Collatz sequence of $2429$.  By Part \ref{startEndPart} of Theorem
  \ref{unwindthm}, $v_{0,0} = 3^n \left( v_{0,n}+1 \right
  )-1$. Rewriting, we get $(v_{0,0}+1)/3^n = v_{0,n}+1$.  Since
  $(2429+1)/3^5 = 10$, and $10 \not \equiv 0 \mod{3}$, we get $v_{0,5}
  = 9$.  By Part \ref{partrelate} of Theorem \ref{unwindthm}, Since $9
  \equiv 0 \mod{3}$, $v_{5,0} \equiv 1 \mod{3}$.  Now $v_{5,0} = 319 =
  3 \times 106 + 1$, therefore. $v_{6,0} = 4 \times 106 + 1 = 425$, by
  Lemma \ref{3n+1lemma}.  Also, since, $106 \equiv 1 \mod{3}$, $v_{6,0}
  \equiv 2 \mod{3}$.  Since $426/3 = 142$, we get $v_{6,0} = 3 \times
  142 -1$, by Part \ref{startEndPart} of Theorem
  \ref{unwindthm}. Since $142 \equiv 1 \mod{3}$, $v_{7,0} \equiv 2
  \mod{3}$, by Lemma \ref{3n+1lemma}.  Continuing this argument, we
  see that subsequence of odd integers of the Reverse Collatz sequence
  of $2429$ fluctuates between numbers that are $\equiv 2 \mod{3}$ and
  $ 1 \mod{3}$, till it reaches $111 \equiv 0 \mod{3}$.

\begin{adjustwidth}{-1in}{.5in}

\[ \begin{array}{llllllllllllllllll}
v_{0,0}= 2429 =  3^5 \times 10 -1 \\
v_{1,0}= 1619  = 3^4 \times 2 \times  10 -1  \\
v_{2,0} = 1079  = 3^3 \times 2^2 \times  10 -1  \\
v_{3,0}=719 = 3^2 \times 2^3 \times  10 -1  \\ 
v_{4,0} = 479  = 3 \times 2^4 \times  10 -1  \\  
v_{5,0} =319  = 2^5 \times 10 -1 = 3 \times 106 + 1 \\
v_{6,0} =425 = 4 \times 106 + 1 = 3 \times 142 -1 \\
v_{7,0} =283 = 4 \times 142 -1 = 3 \times 94 -1 \\
v_{8,0} =377  = 4 \times 94 -1 = 3^3 \times 14 -1   \\
v_{9,0} =251 = 3^2 \times 2 \times 14 -1  \\
v_{10,0} =167 = 3 \times 2^2 \times 14 -1  \\
v_{11,0} = 111= 2^3 \times 14 -1  
\end{array}
\]
\end{adjustwidth}

}
\end{exm}

Thus, we see that the odd numbers of a Reverse Collatz sequence, keep
alternating between numbers that are congruent to $1 \mod{3}$ and $2
\mod{3}$, till it reaches a number that is divisible by $3$.  Which
also means the sequence increases and decreases at regular intervals.
Does this sequence converge? Or does it alternate forever? We cannot
answer this question yet. 

A Reverse Collatz sequence will continue till it reaches a number $A$
that is a multiple of $3$.  Now if $A$ is a multiple of $3$, then
$4A+1 \equiv 1$ mod $3$. So the Reverse Collatz sequence of $4A+1$ is
non trivial. Moreover, the Collatz sequences of $A$ and $4A+1$ are
equivalent. Thus, a Collatz sequence can be extended backwards forever
using jumps as in Example \ref{exmreverse}. A sequence containing
infinite terms is divergent.

\begin{exm} \label{exmreverse} {\em A Collatz sequence can be extended backwards forever using jumps!
 {\supertiny
\[ \begin{array}{rrrrrrrrrrrr}
&&&&  204729 \\ 
&&&& 153547 \\ 
&&&& 230321 \\ 
 &&&&  4 \times 43185 + 1 = 172741 \\
  &&&& \Uparrow \\
&&&&  43185 \\ 
 &&&  8097  & \Rightarrow 4 \times 8097 + 1 = 32389 \\ 
 &&&  6073 \\ 
 &&&  4555 \\ 
 &&&   6833 \\  
&& 1281  
 & \Rightarrow 4 \times 1281 + 1 = 5125 \\ 
&&  961 \\
&&721 \\
  &&  4 \times 135 + 1 = 541  \\
  && \Uparrow \\
 && 135 \\  
 &&   203 \\  
&&  305 \\
& &4 \times 57 + 1 = 229 \\                                           
& & \Uparrow \\
 && 57   \\  
 && 43   \\       
&&  65 \\
&& 49 \\
 &  9   & \Rightarrow 4 \times 9 + 1 = 37  \\
 &   7   \\
 &11   \\                       
  & 17 \\

  &  4 \times 3 + 1 = 13  \\
  & \Uparrow \\
 & 3 &  \\
1 &  \Rightarrow  4 \times 1 + 1 = 5 \\

\end{array}
\] }
} \end{exm}

\begin{thm} \label{falsethm} For any odd integer $A$, there are infinite Collatz sequences  that do not converge. 
\end{thm}
{\em Proof.} Given an odd integer $A$, consider the sequence of jumps
$b_i = 4b_{i-1}+1$ with $b_0=A$. This is an infinite sequence with
equivalent Collatz sequences by Lemma \ref{lemmajump}. If for any $i$,
$b_i \equiv 0$ mod $3$, then $b_{i+1} \equiv 1$ mod $3$, $b_{i+2}
\equiv 2$ mod $3$, and $b_{i+3} \equiv 0$ mod $3$. Which implies there
are infinite jumps for any number $A$ which are not multiples of $3$.
The reverse Collatz sequences of these jumps are non trivial.  Hence,
there are infinite ways to go backwards. Moreover, as in Example
\ref{exmreverse}, these sequences can be extended backwards
infinitely. All these sequences have infinite terms and hence are
divergent. \qed

By Theorem \ref{falsethm}, the Collatz Conjecture is false. End of story.

\end{document}